\documentclass[a4paper,twoside]{article}
\usepackage{a4}
\usepackage{amssymb}
\usepackage{amsmath}
\usepackage{upref}
\usepackage[active]{srcltx}
\usepackage[pagebackref,colorlinks,citecolor=blue,linkcolor=blue]{hyperref}
\usepackage[dvipsnames]{color}
\allowdisplaybreaks[2] 
\newcount\minutes \newcount\hours
\hours=\time
\divide\hours 60
\minutes=\hours
\multiply\minutes -60
\advance\minutes \time
\newcommand{\klockan}{\the\hours:{\ifnum\minutes<10 0\fi}\the\minutes}
\newcommand{\tid}{\today\ \klockan}
\newcommand{\prtid}{\smash{\raise 10mm \hbox{\LaTeX ed \tid}}}
\renewcommand{\prtid}{}
\makeatletter
\def\sectionmark#1{} 
\def\subsectionmark#1{}
\newcommand{\sectnr}{\ifnum \c@secnumdepth >\z@
                 \thesection.\hskip 1em\relax \fi}
\def\@evenhead{\footnotesize\rm\thepage\hfil\leftmark\hfil\llap{\prtid}}
\def\@oddhead{\footnotesize\rm\rlap{\prtid}\hfil\rightmark\hfil\thepage}
\def\tableofcontents{\section*{Contents} 
 \@starttoc{toc}}
\makeatother
\makeatletter
\def\@biblabel#1{#1.}
\makeatother
\makeatletter
\let\Thebibliography=\thebibliography
\renewcommand{\thebibliography}[1]{\def\@mkboth##1##2{}\Thebibliography{#1}
\addcontentsline{toc}{section}{References}
\frenchspacing 
\setlength{\@topsep}{0pt}
\setlength{\itemsep}{0pt}%
\setlength{\parskip}{0pt plus 2pt}%
}
\makeatother
\makeatletter
\def\mdots@{\mathinner.\nonscript\!.%
 \ifx\next,.\else\ifx\next;.\else\ifx\next..\else
 \nonscript\!\mathinner.\fi\fi\fi}
\let\ldots\mdots@
\let\cdots\mdots@
\let\dotso\mdots@
\let\dotsb\mdots@
\let\dotsm\mdots@
\let\dotsc\mdots@
\def\vdots{\vbox{\baselineskip2.8\p@ \lineskiplimit\z@
    \kern6\p@\hbox{.}\hbox{.}\hbox{.}\kern3\p@}}
\def\ddots{\mathinner{\mkern1mu\raise8.6\p@\vbox{\kern7\p@\hbox{.}}%
    \raise5.8\p@\hbox{.}\raise3\p@\hbox{.}\mkern1mu}}
\makeatother
\makeatletter
\def\@seccntformat#1{\csname the#1\endcsname.\quad}
\makeatother
\makeatletter
\let\Enumerate=\enumerate
\renewcommand{\enumerate}{\Enumerate%
\setlength{\@topsep}{0pt}
\setlength{\itemsep}{0pt}%
\setlength{\parskip}{0pt plus 1pt}%
\renewcommand{\theenumi}{\textup{(\alph{enumi})}}%
\renewcommand{\labelenumi}{\theenumi}%
}
\let\endEnumerate=\endenumerate
\renewcommand{\endenumerate}{\endEnumerate\unskip}
\makeatother
\newcommand{\authortitle}[2]{\author{#1}\title{#2}\markboth{#1}{#2}}
\newcommand{\auth}[2]{{#1, #2.}}
\newcommand{\art}[6]{{\sc #1, \rm #2, \it #3 \bf #4 \rm (#5), \mbox{#6}.}}
\newcommand{\book}[3]{{\sc #1, \it #2, \rm #3.}}
\newcommand{\AND}{{\rm and }}
\RequirePackage{amsthm}
\newtheoremstyle{descriptive}%
  {\topsep}   
  {\topsep}   
  {\rmfamily} 
  {}          
  {\bfseries} 
  {.}         
  { }         
  {}          
\newtheoremstyle{propositional}%
  {\topsep}   
  {\topsep}   
  {\itshape}  
  {}          
  {\bfseries} 
  {.}         
  { }         
  {}          
\theoremstyle{propositional}
\newtheorem{thm}{Theorem}[section]
\newtheorem{lem}[thm]{Lemma}
\newtheorem{prop}[thm]{Proposition}
\newtheorem{cor}[thm]{Corollary}
\theoremstyle{descriptive}
\newtheorem{defi}[thm]{Definition}
\newtheorem{remark}[thm]{Remark}
\newtheorem{example}[thm]{Example}
\makeatletter
\renewenvironment{proof}[1][\proofname]{\par
  \pushQED{\qed}%
  \normalfont
  \trivlist
  \item[\hskip\labelsep
        \itshape
    #1\@addpunct{.}]\ignorespaces
}{%
  \popQED\endtrivlist\@endpefalse
}
\makeatother
\newcommand{\setm}{\setminus}
{\catcode`p =12 \catcode`t =12 \gdef\eeaa#1pt{#1}}      
\def\accentadjtext#1{\setbox0\hbox{$#1$}\kern   
                \expandafter\eeaa\the\fontdimen1\textfont1 \ht0 }
\def\accentadjscript#1{\setbox0\hbox{$#1$}\kern 
                \expandafter\eeaa\the\fontdimen1\scriptfont1 \ht0 }
\def\accentadjscriptscript#1{\setbox0\hbox{$#1$}\kern   
                \expandafter\eeaa\the\fontdimen1\scriptscriptfont1 \ht0 }
\def\accentadjtextback#1{\setbox0\hbox{$#1$}\kern       
                -\expandafter\eeaa\the\fontdimen1\textfont1 \ht0 }
\def\accentadjscriptback#1{\setbox0\hbox{$#1$}\kern     
                -\expandafter\eeaa\the\fontdimen1\scriptfont1 \ht0 }
\def\accentadjscriptscriptback#1{\setbox0\hbox{$#1$}\kern 
                -\expandafter\eeaa\the\fontdimen1\scriptscriptfont1 \ht0 }
\def\itoverline#1{{\mathsurround0pt\mathchoice
        {\rlap{$\accentadjtext{\displaystyle #1}
                \accentadjtext{\vrule height1.593pt}
                \overline{\phantom{\displaystyle #1}
                \accentadjtextback{\displaystyle #1}}$}{#1}}
        {\rlap{$\accentadjtext{\textstyle #1}
                \accentadjtext{\vrule height1.593pt}
                \overline{\phantom{\textstyle #1}
                \accentadjtextback{\textstyle #1}}$}{#1}}
        {\rlap{$\accentadjscript{\scriptstyle #1}
                \accentadjscript{\vrule height1.593pt}
                \overline{\phantom{\scriptstyle #1}
                \accentadjscriptback{\scriptstyle #1}}$}{#1}}
        {\rlap{$\accentadjscriptscript{\scriptscriptstyle #1}
                \accentadjscriptscript{\vrule height1.593pt}
                \overline{\phantom{\scriptscriptstyle #1}
                \accentadjscriptscriptback{\scriptscriptstyle #1}}$}{#1}}}}
\def\itunderline#1{{\mathsurround0pt\mathchoice
        {\rlap{$\underline{\phantom{\displaystyle #1}
                \accentadjtextback{\displaystyle #1}}$}{#1}}
        {\rlap{$\underline{\phantom{\textstyle #1}
                \accentadjtextback{\textstyle #1}}$}{#1}}
        {\rlap{$\underline{\phantom{\scriptstyle #1}
                \accentadjscriptback{\scriptstyle #1}}$}{#1}}
        {\rlap{$\underline{\phantom{\scriptscriptstyle #1}
                \accentadjscriptscriptback{\scriptscriptstyle #1}}$}{#1}}}}
\newcommand{\Cpw}{C_{p,w}}
\newcommand{\grad}{\nabla}
\DeclareMathOperator{\Div}{div}
\DeclareMathOperator{\supp}{supp}
\DeclareMathOperator*{\essliminf}{ess\,lim\,inf}
\DeclareMathOperator*{\essinf}{ess\,inf}
\newcommand{\bdy}{\partial}
\newcommand{\loc}{_{\rm loc}}
\renewcommand{\emptyset}{\varnothing}
\newcommand{\A}{{\mathcal A}}
\newcommand{\U}{{\mathcal U}}
\newcommand{\LL}{{\mathcal L}}
\newcommand{\al}{\alpha}
\newcommand{\be}{\beta}
\newcommand{\Om}{\Omega}
\renewcommand{\phi}{\varphi}
\newcommand{\eps}{\varepsilon}
\newcommand{\p}{{$p\mspace{1mu}$}}
\newcommand{\R}{\mathbf{R}}
\newcommand{\C}{\mathbf{C}}
\newcommand{\K}{{\mathcal K}}
\newcommand{\vb}{\bar{v}}
\newcommand{\ut}{\tilde{u}}
\newcommand{\uP}{\itoverline{P}} 
\newcommand{\lP}{\itunderline{P}} 
\newcommand{\clOm}{\overline{\Omega}}
\newcommand{\Hp}{H^{1,p}}
\newcommand{\Hploc}{H^{1,p}\loc}
\numberwithin{equation}{section}
\newenvironment{ack}{\medskip{\it Acknowledgement.}}{}
\begin{document}
\authortitle{Anders Bj\"orn, Jana Bj\"orn  and Abubakar Mwasa}
         {Resolutivity and invariance for the Perron method for degenerate equations}
\title{Resolutivity and invariance for the Perron method for degenerate equations of divergence type}
\author{
Anders Bj\"orn\\
\it\small Department of Mathematics, Link\"oping University, SE-581 83 Link\"oping, Sweden\\
\it \small anders.bjorn@liu.se, ORCID\/\textup{:} 0000-0002-9677-8321
\\
\\
Jana Bj\"orn\\
\it\small Department of Mathematics, Link\"oping University, SE-581 83 Link\"oping, Sweden\\
\it \small jana.bjorn@liu.se, ORCID\/\textup{:} 0000-0002-1238-6751
\\ 
\\
Abubakar Mwasa \\
\it\small Department of Mathematics, Link\"oping University, SE-581 83 Link\"oping, Sweden\\
\it\small Department of Mathematics, Busitema University, P.O.Box 236, Tororo, Uganda \\
\it \small abubakar.mwasa@liu.se, a.mwasa@yahoo.com,
ORCID\/\textup{:} 0000-0003-4077-3115
}
\date{}

\maketitle

\noindent{\small
\begin{abstract}
\noindent
We consider Perron solutions to the Dirichlet problem for the
quasilinear elliptic equation $\Div\A(x,\grad u) = 0$
in a bounded
open set $\Om\subset\R^n$.
The vector-valued function $\A$ satisfies the standard ellipticity 
assumptions with a parameter $1<p<\infty$ and a \p-admissible weight $w$.
We show that arbitrary perturbations on sets of $(p,w)$-capacity zero of 
continuous (and certain quasicontinuous) boundary data~$f$
are resolutive and that the Perron solutions
for $f$ and such perturbations coincide. 
As a consequence, we prove that the Perron solution
with continuous boundary data
is the unique bounded solution that takes the required
boundary data outside a set of $(p,w)$-capacity zero. 
\end{abstract}
}

\bigskip


\medskip


\section{Introduction}
We consider the Dirichlet problem for quasilinear elliptic equations of the form
\begin{equation}			\label{Div-A}
\Div\A(x,\grad u) = 0
\end{equation}
in a bounded nonempty open subset $\Om$ of the $n$-dimensional Euclidean space $\R^n$. 
The mapping $\A:\Om\times\R^n\rightarrow \R^n$ satisfies the 
standard ellipticity assumptions with a parameter $1<p<\infty$ and 
a \p-admissible weight as in 
Heinonen--Kilpel\"ainen--Martio~\cite[Chapter~3]{HKM}. 

The Dirichlet problem 
amounts to
finding a solution of the partial differential equation in $\Om$ with prescribed 
boundary data on the boundary of $\Om$.
One of the most useful approaches to solving the Dirichlet problem in $\Om$ 
with arbitrary boundary data $f$ is  the Perron method.
This method was introduced by Perron~\cite{P} 
and independently Remak~\cite{remak}
in 1923 for the Laplace equation $\Delta u=0$ in a bounded domain $\Om\subset\R^n$.
It gives an upper and a lower Perron solution (see Definition~\ref{Perron-defi}) 
and when the two coincide, we get a suitable solution $Pf$ of the Dirichlet problem and $f$ is called \emph{resolutive}.

The Perron method for linear equations in Euclidean domains was studied 
by Brelot~\cite{MB1}, where a complete characterization of resolutive functions was given
in terms of the harmonic measure.
The Perron method was later
extended to nonlinear equations. 
Granlund--Lindqvist--Martio~\cite{GLM} were the first to use 
the Perron method to study the nonlinear equation 
\[ 
\Div(\grad_qF(x,\grad u))=0
\] 
(where $\grad_qF$ stands for the gradient of $F$ with respect to the second variable).
This
is a special type of equation~\eqref{Div-A}, including the \p-Laplace equation 
\begin{equation} \label{eq-p-Lap}
\Delta_pu:=\Div(|\grad u|^{p-2}\grad u)=0.
\end{equation}
Lindqvist--Martio~\cite{LM} studied boundary regularity of 
\eqref{Div-A} in the unweighted case
and also showed that continuous boundary data
$f$ are resolutive when $p>n-1$.
Kilpel\"ainen~\cite{TK89} extended the resolutivity to general $p$,
which in turn was extended to weighted $\R^n$ by Heinonen--Kilpel\"ainen--Martio~\cite{HKM}.
More recently, the Perron method was used to study \p-harmonic functions 
in the metric setting, see
\cite{BB}--\cite{BBSjodin}.
  
In this paper, we consider the weighted equation
\begin{equation*}			
\Div\A(x,\grad u) = 0
\end{equation*}
and show that arbitrary perturbations on sets of $(p,w)$-capacity zero 
of continuous boundary data~$f$
are resolutive and that the Perron solution for $f$ 
and such perturbations coincide, see Theorem~\ref{thm-Pf+h-Pf}. 
In Proposition~\ref{prop-Pf+h-Pf-Hf}, we also obtain, as a by-product, 
that Perron solutions of perturbations 
of Lipschitz boundary data $f$ are the same as the Sobolev solution of $f$.  
This perturbation result, as well as the equality of the Perron and Sobolev solutions, holds also 
for quasicontinuous representatives of Sobolev functions, see
Theorem~\ref{thm-Sob-Pf=Hf}.

Moreover, we prove in Theorem~\ref{thm-unique}
that the Perron solution for the equation \eqref{Div-A}
with continuous boundary data
is the \emph{unique} bounded solution of \eqref{Div-A}
that takes the required 
boundary data outside a set of $(p,w)$-capacity zero. 
A somewhat weaker uniqueness result is proved for quasicontinuous Sobolev functions in 
Corollary~\ref{cor-unique-qcont}.

Much as we use Heinonen--Kilpel\"ainen--Martio~\cite{HKM} as the principal literature for this paper, our proof of resolutivity for continuous boundary data is 
quite
different from the one considered in \cite{HKM}.
In particular, we do not use exhaustions by regular domains. 
The obstacle problem for the operator $\Div\A(x,\grad u)$
and a convergence theorem for obstacle problems play a crucial role in the proof of our main results.

For \p-harmonic functions, i.e.\ solutions
of the \p-Laplace equation~\eqref{eq-p-Lap}, most of the
results in this paper follow from
Bj\"orn--Bj\"orn--Shan\-mu\-ga\-lin\-gam~\cite{BBS2}, \cite{BBSdir},
where this was proved for \p-energy minimizers in metric spaces.
The proofs here have been inspired by \cite{BBS2} and \cite{BBSdir},
but have been adapted to the usual Sobolev spaces to make them
more accessible for people not familiar
with the nonlinear potential theory on metric spaces
and Sobolev spaces based on upper gradients.
They also apply to the more general $\A$-harmonic functions,
defined by equations rather than minimization problems.

\begin{ack}
A.~B and J.~B. were partially supported by the Swedish Research Council grants 
2016-03424 resp.\ 621-2014-3974 and
2018-04106.
A.~M. was supported by 
the SIDA (Swedish International Development
Cooperation Agency)
project 316-2014  
``Capacity building in Mathematics and its applications'' 
under the SIDA bilateral program with the Makerere University 2015--2020,
contribution No.\ 51180060.
\end{ack}

\section{Notation and preliminaries}
\label{sect-not-pre}
In this section, we present the basic notation and definitions that will be needed in this paper.
Throughout, we assume that $\Om$ is a bounded nonempty open subset of the $n$-dimensional Euclidean space $\R^n, n\geq2$, and $1<p<\infty$.
We use $\bdy\Om$ and $\clOm$ to denote the boundary and the closure 
of $\Om$, respectively.

We write $x$ to mean a point $x=(x_1,\cdots,x_n)\in\R^n$ and for a function $v$ 
which is infinitely many times continuously
differentiable, i.e.\ $v\in C^\infty(\Om)$, 
we write $\grad v=(\bdy_1v,\cdots,\bdy_nv)$ for the gradient of $v$.
We follow Heinonen--Kilpel\"ainen--Martio~\cite{HKM} as the primary reference 
for the material in this paper.

First, we give the definition of a weighted Sobolev space, 
which is 
crucial  when studying degenerate elliptic differential equations, see 
\cite{HKM} and Kilpel\"ainen~\cite{TK94}.

\begin{defi}
The \emph{weighted Sobolev space \(\Hp(\Om,w)\)} is defined to be the completion of 
the set of all 
\(v\in C^\infty(\Om)\) such that 
\[
\|v\|_{\Hp(\Om,w)}=\biggl(\int_\Om(|v|^p+|\grad v|^p) w
\,dx\biggr)^{1/p}<\infty
\]
with respect to the norm \(\|v\|_{\Hp(\Om,w)}\), where \(w\) is the weight 
function which we define later.
\end{defi}

The space \(\Hp_0(\Om,w)\) is the completion of \(C_0^\infty(\Om)\) 
in \(\Hp(\Om,w)\) while a function $v$ is in \(\Hploc(\Om,w)\) 
if and only if it belongs to \(\Hp(\Om',w)\) for every open set \(\Om'\Subset\Om\).
As usual, $E \Subset \Om$ if $\itoverline{E}$ is a compact subset of $\Om$
and
\[
  C_0^\infty(\Om)=\{v \in C^\infty(\R^n) : \supp v \Subset \Om\}.
\]

Throughout the paper, the mapping $\A:\Om\times\R^n\to\R^n$, defining
the elliptic operator \eqref{Div-A}, satisfies the following
assumptions with a parameter $1<p<\infty$, a \p-admissible weight
$w(x)$ and for some constants $\al,\be>0$, see \cite[(3.3)--(3.7)]{HKM}:

First, assume that $\A(x,q)$ is measurable in $x$ for every $q\in\R^n$, and continuous in $q$ for a.e.\ $x\in\R^n$.
Also, for all $q\in\R^n$ and a.e.\ $x\in\R^n$, the following hold
\begin{alignat}{2}   \label{ellip-conds}
	 \A(x,q)\cdot q \ge\al w(x)|q|^p& &\quad &\text{and} \quad
	|\A(x,q)|\le\be w(x)|q|^{p-1},\\
	(\A(x,q_1)-\A(x,q_2))\cdot(q_1-q_2) >0 &&\quad&\text{for }
        q_1,q_2\in\R^n,\  q_1\neq q_2, \nonumber \\
        \A(x,\lambda q)=\lambda|\lambda|^{p-2}\A(x,q) &&\quad
        &\text{for }\lambda\in\R,\ \lambda\neq0. \nonumber
\end{alignat}

\begin{defi}
  A function $u\in\Hploc(\Om,w)$ is said to be a
(\emph{weak}) \emph{solution} of \eqref{Div-A} in $\Om$
if for all test functions $\phi\in C_0^\infty(\Om)$,
the following integral identity holds
\begin{equation}      \label{Int-A}
\int_\Om\A(x,\grad u)\cdot\grad\phi\,dx=0.
\end{equation}
A function $u\in\Hploc(\Om,w)$ is said to be a
\emph{supersolution} of (\ref{Div-A}) in $\Om$ if 
for  all nonnegative functions $\phi\in C_0^\infty(\Om)$,
\[
\int_\Om\A(x,\grad u)\cdot\grad\phi\,dx\geq 0.
\]
A function $u$ is a \emph{subsolution} of \eqref{Div-A} if $-u$ is a supersolution of (\ref{Div-A}). 
\end{defi}

The sum of two (super)solutions is in general not a (super)solution. However, 
if $u$ and $v$ are two (super)solutions, then $\min\{u,v\}$ is a supersolution, 
see
\cite[Theorem~3.23]{HKM}.
If $u$ is a supersolution and $a,b\in\R$, then $a u+b$
is a supersolution provided that $a \ge 0$.

It is rather straightforward that $u$ is a solution if and only
  if it is both a sub- and a supersolution, see \cite[bottom p.~58]{HKM}. 
By \cite[Theorems~3.70 and~6.6]{HKM}, every
solution $u$ has a H\"older continuous representative $v$
(i.e.\ $v=u$ a.e.).

\begin{defi}
A function $u$ is $\A$-\emph{harmonic} in $\Om$ if it is a continuous weak solution 
of \eqref{Div-A} in $\Om$.
\end{defi}

We remark that $\A$-harmonic functions do not in general form a linear space.
However, if $u$ is $\A$-harmonic and $a,b\in\R$, then $a u+b$
is also $\A$-harmonic.
Nonnegative $\A$-harmonic functions $u$ in 
a connected open set
$\Om$ satisfy Harnack's inequality $\sup_K u\leq c\inf_K u$ whenever 
$K\subset\Om$ is compact, with the constant $c$ depending on $K$, 
see \cite[Section~6.2]{HKM}.

\begin{defi}
A \emph{weight}  $w$ on $\R^n$ is a nonnegative locally integrable function.
We say that a weight $w$ is 
\emph{\p-admissible} with \(p\geq1\) if the associated measure 
$d\mu = w\,dx$
is doubling and supports a \p-Poincar\'e inequality,
see \cite[Chapters~1 and~20]{HKM}.
\end{defi}

For instance, weights belonging to the Muckenhoupt class
$A_p$ are \p-admissible
as exhibited for example 
by Heinonen--Kilpel\"ainen--Martio~\cite{HKM} and  Kilpel\"ainen~\cite{TK94}.
By a weight $w \in A_p$ we mean that there exists a constant \(C>0\) such that
for all balls $B\subset\R^n$,  
\begin{alignat*}{2}
  \biggl(\int_B w(x)\,dx\biggr)\biggl(\int_B w(x)^{1/(1-p)}\,dx\biggr)^{p-1}
  &\leq C|B|^p, & \quad & \text{if } 1<p<\infty, \\
  \int_B w(x)\,dx& \leq C |B|\essinf_B w, & \quad & \text{if } p=1, 
\end{alignat*}
where 
\(|B|\) is the \(n\)-dimensional Lebesgue
measure of \(B\). 

We follow \cite[Section~2.35]{HKM} 
defining the Sobolev capacity as follows.

\begin{defi}		\label{defi-wcap}
Let $E$ be a subset of $\R^n$. The \emph{Sobolev $(p,w)$-capacity} of $E$ is 
\[
\Cpw(E)=\inf\int_{\R^n}(|u|^p+|\grad u|^p)w\,dx,
\]
where the infimum is taken over all $u\in \Hp(\R^n,w)$
such that $u=1$ in an open set containing $E$.
\end{defi}

The Sobolev \((p,w)\)-capacity is a monotone, subadditive set
function.
It follows directly from the definition that for all $E\subset\R^n$,
\begin{equation} \label{eq-Cpw-outer}
\Cpw(E)=\inf_{\substack{ G\supset E \\ G \text{ open}}}  \Cpw(G).
\end{equation}
In particular, if $\Cpw(E)=0$ then there exist open sets $U_j\supset E$
with $\Cpw(U_j)\to0$ as $j\to\infty$.
For details, we refer the interested reader to 
\cite[Section~2.1]{HKM}.
A property is said to hold \emph{quasieverywhere} (abbreviated q.e.), 
if it holds for every point outside
a set of Sobolev $(p,w)$-capacity zero.

\section{Perron solutions and resolutivity}	\label{sect-Perron-res}

In order to discuss the Perron solutions for \eqref{Div-A}, we first recall 
the following basic results from 
Heinonen--Kilpel\"ainen--Martio~\cite[Chapters~7 and~9]{HKM}.

\begin{defi}\label{A-superh-defi}
A function $u:\Om\rightarrow (-\infty,\infty]$ is $\A$-\emph{superharmonic} in $\Om$ if 
\begin{enumerate}
\renewcommand{\theenumi}{\textup{(\roman{enumi})}}%
\item $u$ is lower semicontinuous,
\item $u$ is not identically $\infty$ in any component of $\Om$,
\item for every open $\Om'\Subset\Om$ and all functions \(v\in C(\clOm')\) 
which are $\A$-harmonic in \(\Om'\), we have $v\leq u$ in $\Om'$ 
whenever $v\leq u$ on \(\bdy\Om'\).
\end{enumerate}
A function $u:\Om\rightarrow [-\infty,\infty)$ is $\A$-\emph{subharmonic} 
in $\Om$ if $-u$ is $\A$-superharmonic in~$\Om$.
\end{defi}

Let $u$ and $v$ be $\A$-superharmonic. 
Then $a u+b$ and $\min\{u,v\}$ are $\A$-super\-har\-mon\-ic
whenever $a\geq0$ and $b$ are real numbers,
but in general $u+v$ is not $\A$-superharmonic, see \cite[Lemmas~7.1~and~7.2]{HKM}.

We briefly state how supersolutions and $\A$-superharmonic functions are related.
It is proved in~\cite[Theorem~7.16]{HKM} that if $u$ is a supersolution of \eqref{Div-A}
and
\begin{equation} \label{eq-u*}
  u^*(x)=\essliminf_{\Om\ni y\rightarrow x}u(y)\quad\text{for every }x\in\Om,
\end{equation}
then $u^*=u$ a.e.\ and $u^*$ is  \(\A\)-super\-har\-mon\-ic.
Conversely, if 
$u$ is an  \(\A\)-super\-har\-mon\-ic function in  \(\Om\),
  then $u^*=u$ in $\Om$.
  If moreover, $u$ is
locally bounded from above, then
\(u\in\Hploc(\Om,w)\) and  \(u\) is a supersolution of \eqref{Div-A} in  \(\Om\), 
see \cite[Corollary~7.20]{HKM}. 
That is, every supersolution has an $\A$-superharmonic representative 
and locally bounded $\A$-superharmonic functions are supersolutions.

\begin{defi} \label{Perron-defi}
Given  a function $f:\bdy\Om\rightarrow [-\infty,\infty]$, let $\U_f$ be the set of all $\A$-superharmonic functions $u$ on $\Om$ bounded from below such that
\[
\liminf_{\Om\ni y\rightarrow x}u(y)\geq f(x)\quad\text{for all } x\in\bdy\Om.
\]
The \emph{upper Perron solution} $\uP f$ of $f$ is defined by 
\[
\uP f(x)=\inf_{u\in\U_f}u(x),\quad x\in\Om.
\]
Analogously, let $\LL_f$ be the set of all $\A$-subharmonic functions $v$ on $\Om$ 
bounded from above such that
\[
\limsup_{\Om\ni y\rightarrow x}v(y)\leq f(x)\quad\text{for all }x\in\bdy\Om.
\]
The \emph{lower Perron solution} $\lP  f$ of $f$ is defined by 
\[
\lP  f(x)=\sup_{v\in\LL_f}v(x),\quad x\in\Om.
\]
\end{defi} 

We remark that if $\U_f=\emptyset$, then
$\uP f\equiv\infty$ and if $\LL_f=\emptyset$, then $\lP f\equiv-\infty$. 
In every component $\Om'$ of \(\Om\), $\uP f$ (and $\lP f$)
is either \(\A\)-harmonic or identically \(\pm\infty\) in \(\Om'\),
see \cite[Theorem~9.2]{HKM}.

If $\uP f=\lP f$
is $\A$-harmonic,
then $f$ is said to be \emph{resolutive} with respect to $\Om$. 
In this case, we write $Pf:=\uP f$.
Continuous functions $f$ are resolutive by 
\cite[Theorem~9.25]{HKM}.

The following comparison principle shows that $\lP f\leq \uP f$.

\begin{thm}		\label{thm-comp-princ}
\textup{(\cite[Comparison principle~7.6]{HKM})}
Assume that $u$ is $\A$-superharmonic and that $v$ is $\A$-subharmonic in $\Om$. 
If 
\[
\limsup_{\Om\ni y\rightarrow x}v(y)\leq \liminf_{\Om\ni y\rightarrow x}u(y)\quad\text{for all } x\in\bdy\Om,
\]
and if both sides are not simultaneously $\infty$ or $-\infty$, then $v\leq u$ in $\Om$.
\end{thm}

We follow \cite[Chapter~3]{HKM} giving the following definition.

\begin{defi}
Let $\psi:\Om\rightarrow [-\infty,\infty]$ and $f\in\Hp(\Om,w)$. 
Let
\[
\K_{\psi,f}(\Om)=\{v\in\Hp(\Om,w):v-f\in\Hp_0(\Om,w)\text{ and } v\geq\psi\text{ a.e\ in }\Om\}.
\]
A function $u\in\K_{\psi,f}(\Om)$ is a
\emph{solution of the obstacle problem} in $\Om$ 
with obstacle $\psi$ and boundary data $f$ if
\[ 
\int_\Om\A(x, \grad u)\cdot\grad(v-u)\,dx\geq 0\quad\text{for all }v\in\K_{\psi,f}(\Om).
\]
\end{defi}

In particular, a solution $u$ of the obstacle problem for
  $\K_{\psi,u}(\Om)$ 
with $\psi\equiv -\infty$ is a solution of \eqref{Div-A}.
By considering $v=u+\phi$ with \(0\leq\phi\in C^\infty_0(\Om)\), it is 
easily seen that the solution $u$ of the obstacle problem is always a supersolution 
of \eqref{Div-A} in $\Om$.
Conversely, a supersolution \(u\) in $\Om$ is always a solution of 
the obstacle problem for $\K_{u,u}(\Om')$ for all open sets $\Om' \Subset \Om$.
Moreover, a solution \(u\) of \eqref{Div-A} is a solution of 
the obstacle problem for \(\K_{\psi,u}(\Om')\) with $\psi\equiv-\infty$ 
for all $\Om'\Subset\Om$, see \cite[Section~3.19]{HKM}.
 
By \cite[Theorem~3.21]{HKM}, there is an almost everywhere (a.e) 
unique solution $u$ of 
the obstacle problem whenever $\K_{\psi,f}(\Om)$ is nonempty.
Furthermore, by defining  $u^*$ as in \eqref{eq-u*}
we get a \emph{lower semicontinuously regularized solution} in the same 
equivalence class as $u$, see \cite[Theorem~3.63]{HKM}. 

We call $u^*$ the
\emph{lower semicontinuous\/ \textup{(}lsc\/\textup{)} regularization}
of $u$.
Moreover, with $\psi\equiv-\infty$, the lsc-regularization of the solution 
of the obstacle problem for \(\K_{\psi,f}(\Om)\) provides us with the 
\(\A\)-\emph{harmonic extension} \(Hf\) of \(f\) in \(\Om\), that is,  
\(Hf-f\in\Hp_0(\Om,w)\)
and $Hf$ is $\A$-harmonic.
The continuity of \(Hf\) in \(\Om\) is guaranteed by \cite[Theorem~3.70]{HKM}.

\begin{defi}
A point $x\in\bdy\Om$ is \emph{Sobolev regular} if, 
for every $f\in\Hp(\Om,w)\cap C(\clOm)$, the $\A$-harmonic function $Hf$ 
in $\Om$ with $Hf-f\in\Hp_0(\Om,w)$ satisfies
\[
\lim_{\Om\ni y\rightarrow x}Hf(y)=f(x).
\]
Furthermore, $x  \in \bdy \Om$ is \emph{regular} if 
\[
    \lim_{\Om\ni y\rightarrow x}Pf(y)=f(x)
    \quad \text{for all } f \in C(\bdy \Om).
\]
If $x\in\bdy\Om$ is not (Sobolev) regular, then it is (Sobolev) 
\emph{irregular}.
\end{defi}

By \cite[Theorem~9.20]{HKM}, $x$ is regular if and only if it is Sobolev regular,
we will therefore just say ``regular'' from now on.
By the Kellogg property~\cite[Theorem~8.10 and~9.11]{HKM}, 
the set of irregular points on $\bdy\Om$ has Sobolev $(p,w)$-capacity zero.

The following result is due to Bj\"orn--Bj\"orn--Shanmugalingam~\cite[Lemma~5.3]{BBS2}.
Here it is slightly modified to suite our context.
For completeness and the reader's convenience, the proof is included.

\begin{lem}		\label{lem-small-Cp}
  Let $\{U_k\}_{k=1}^\infty$ be a decreasing sequence of 
  open sets 
in $\R^n$ such that $\Cpw(U_k)<2^{-kp}$. 
Then there exists a decreasing sequence of nonnegative functions $\{\psi_j\}_{j=1}^\infty$ 
such that for all $j,m=1,2,\cdots$\,,
\[
\|\psi_j\|_{\Hp(\R^n,w)}<2^{-j}\quad\text{and} \quad \psi_j\geq m\text{ in }U_{j+m}.
\]
In particular, $\psi_j=\infty$ on $\bigcap_{k=1}^\infty U_k$.
\end{lem}

\begin{proof}
Since $\Cpw(U_k)<2^{-kp}$, by Definition~\ref{defi-wcap} 
there exist $\phi_k\in \Hp(\R^n,w)$ 
such that $\phi_k=1$ in $U_k$ 
and $\|\phi_k\|_{\Hp(\R^n,w)}<2^{-k}$.
Replacing $\phi_k$ by its positive part $\max\{\phi_k,0\}$, 
we can assume that each $\phi_k$ is nonnegative.
Define 
\[
\psi_j=\sum_{k=j+1}^\infty\phi_k,\quad j=1,2,\cdots.
\] 
Then 
\[
\|\psi_j\|_{\Hp(\R^n,w)}\leq\sum_{k=j+1}^\infty\|\phi_k\|_{\Hp(\R^n,w)}
< \sum_{k=j+1}^\infty2^{-k} = 2^{-j}.
\] 
Since $\phi_k\geq1$ on each $U_k$ and $U_k\supset U_{j+m}$ when $j+1\leq k\leq j+m$, 
it follows that $\psi_j\geq m$ in $U_{j+m}$.
\end{proof}

We will need the following convergence theorem due to 
Heinonen--Kilpel\"ainen--Martio~\cite[Theorem~3.79]{HKM}
 in order to prove the next proposition.

\begin{thm}  	\label{cov-thm}
Let $\{\psi_j\}_{j=1}^\infty$ be an a.e.\ decreasing sequence of functions in $\Hp(\Om,w)$ 
such that $\psi_j\rightarrow\psi$ in $\Hp(\Om,w)$.
Let \(u_j\in\Hp(\Om,w)\) be a solution of the obstacle problem for
\(\K_{\psi_j,\psi_j}(\Om)\).
Then there exists a function \(u\in\Hp(\Om,w)\) such that the sequence \(u_j\) 
decreases a.e.\ in $\Om$ to \(u\) and \(u\) is a solution 
of the obstacle problem for \(\K_{\psi,\psi}(\Om)\).
\end{thm}

\begin{prop} \label{prop-Pf+h-Pf-Hf}
Let the function $f$ be Lipschitz on $\overline{\Om}$ and
$h:\bdy\Om\to [-\infty,\infty]$ 
be such that $h=0$ q.e.\ on $\bdy\Om$. 
Then
both $f$ and $f+h$ are resolutive and
  \[
  P(f+h)=Pf=Hf.
\]
\end{prop}

\begin{proof}
Since $f$ is Lipschitz and \(\Om\) bounded, we get that $f\in\Hp(\Om, w)$.
First, we assume that $f\geq0$.
Let $I_p\subset\bdy\Om$ be the set of all irregular points. 
Let $E=\{x\in\bdy\Om:h(x)\neq0\}$.
Then by the Kellogg property \cite[Theorem~8.10]{HKM}, we have $\Cpw(I_p\cup E)=0$.
Using \eqref{eq-Cpw-outer}, we can find a decreasing sequence
$\{U_k\}_{k=1}^\infty$
of bounded open sets 
in $\R^n$ such that $I_p\cup E\subset U_k$ and $\Cpw(U_k)<2^{-kp}$.
Consider the decreasing sequence of nonnegative functions $\{\psi_j\}_{j=1}^\infty$ given in Lemma~\ref{lem-small-Cp}. 

Let $u_j$ be the lsc-regularized solution of the obstacle problem
with obstacle and boundary data
$f_j=Hf+\psi_j$, see \cite[Theorems~3.21~and~3.63]{HKM}.
Let $m$ be a positive integer. 
By the comparison principle \cite[Lemma~3.18]{HKM}, we have that $Hf\geq0$ and hence by Lemma~\ref{lem-small-Cp},
\[
f_j=Hf+\psi_j\geq\psi_j\geq m\quad\text{in }U_{j+m}\cap\Om.
\]
In particular, $u_j\geq f_j\geq m$ a.e.\ in $U_{j+m}\cap\Om$ and since $u_j$ is 
lsc-regularized, we have that 
\begin{equation} \label{eq-uj-m}
u_j\geq m\quad\text{everywhere in }U_{j+m}\cap\Om.
\end{equation}
Let $\eps>0$ and $x\in\bdy\Om$ be arbitrary. 
If $x\notin U_{j+m}$, then $x$ is a regular point and thus $Hf$ is continuous at $x$. Hence, there is a neighbourhood $V_x$ of $x$ such that  
\[
Hf(y)\geq f(x)-\eps=(f+h)(x)-\eps\quad\text{for all } y\in V_x\cap\Om.
\]
As
$\psi_j\geq0$, we have that $f_j=Hf+\psi_j\geq Hf$. 
So, 
\[
u_j(y)\geq f_j(y)\geq (f+h)(x)-\eps \quad \text{for a.e.\ }y\in V_x\cap\Om.
\]
Since $u_j$ is lsc-regularized, we get 
\[
u_j(y)\geq (f+h)(x)-\eps\quad\text{for all }y\in V_x\cap \Om.
\] 
And if $x\in U_{j+m}$, we instead 
let $V_x=U_{j+m}$. Then $u_j\ge m$ in $V_x\cap\Om$ by \eqref{eq-uj-m}.

Consequently, for all $x\in\bdy\Om$, we have
\[
u_j(y)\geq\min\{(f+h)(x)-\eps,m\}\quad\text{for all }y\in V_x\cap\Om.
\]
Thus, 
\begin{equation}   \label{eq-min-f+h-m}
\liminf_{\Om\ni y\rightarrow x}u_j(y)\geq\min\{(f+h)(x)-\eps,m\}
\quad\text{for all }x\in\bdy\Om.
\end{equation}
Letting $\eps\rightarrow 0$ and $m\rightarrow\infty$ yields
\[
\liminf_{\Om\ni y\rightarrow x}u_j(y)\geq (f+h)(x)\quad\text{for all }x\in\bdy\Om.
\]

Since $u_j$ is $\A$-superharmonic and nonnegative, we conclude that $u_j\in\mathcal{U}_{f+h}(\Om)$, and thus $u_j\geq\uP (f+h)$. 
As
$Hf$ is the solution of the obstacle problem for $\K_{Hf,Hf}(\Om)$, 
we get by Theorem~\ref{cov-thm} that the sequence $u_j$ decreases a.e.\ to $Hf$ in $\Om$. 
Thus, $Hf\geq \uP(f+h)$ a.e.\ in $\Om$. 
But $Hf$ and $\uP(f+h)$ are continuous, so we have that for all Lipschitz functions $f\ge0$,
\begin{equation}	\label{Ineq-Hf-Pf+h}
Hf\ge\uP(f+h)\quad\text{everywhere in }\Om.
\end{equation} 

Next, let $f$ be an arbitrary Lipschitz function on $\overline{\Om}$.
Since $f$ is bounded, there exists a constant $c\in\R$ such that $f+c\geq0$.
By the definition of $Hf$ and of Perron solutions 
we see that
\[
    H(f+c)=Hf+c 
\quad \text{and} \quad
  \uP(f+h+c) = \uP(f+h)+c.
\]
This  together with \eqref{Ineq-Hf-Pf+h} shows that
\[
Hf =H(f+c)-c \ge \uP(f+h+c) -c = \uP (f+h),
\] 
i.e.\ \eqref{Ineq-Hf-Pf+h} holds for arbitrary Lipschitz functions $f$.
Applying it to $-f$ and $-h$ gives us that
\[
  Hf =-H(-f) \le - \uP (-f-h) = \lP (f+h).
\]
Together with the inequality $\lP (f+h) \le \uP(f+h)$, 
implied by Theorem~\ref{thm-comp-princ},  we get that
\[
    Hf \le \lP (f+h) \le \uP(f+h) \le Hf,
\]
and thus $P(f+h)=Hf$ and $f+h$ is resolutive.
Finally, letting $h=0$, it follows directly that $f$ is resolutive and $Pf=Hf$.
\end{proof}

It is now possible to extend the resolutivity results to continuous functions. This gives us an alternative way of solving the Dirichlet problem with prescribed continuous boundary data.

\begin{thm}		\label{thm-Pf+h-Pf}
Let $f\in C(\bdy\Om)$ and $h:\bdy\Om\rightarrow [-\infty,\infty]$ be such that
$h=0$ q.e.\ on $\bdy\Om$. Then both $f$ and $f+h$ are resolutive and
$P(f+h)=Pf$.
\end{thm}

\begin{proof}
Since continuous functions can be approximated uniformly by Lipschitz functions, we have that there exists a sequence $\{f_k\}_{k=1}^\infty$ of Lipschitz functions such that 
\begin{equation} \label{eq-fk-f}
f_k-2^{-k}\le f\le f_k+2^{-k} \quad\text{on } \bdy\Om.
\end{equation}
From Definition~\ref{Perron-defi} it follows that
\[
\uP f_k-2^{-k}\le\uP f\le\uP f_k+2^{-k} \quad\text{in }\Om,
\]
i.e.\  
the functions $\uP f_k$ converge uniformly to $\uP f$ in $\Om$ as $k\to\infty$.
Using \eqref{eq-fk-f}, we also obtain similar inequalities for $\lP f,\uP(f+h)$ and $\lP(f+h)$ in terms of $\lP f_k, \ \uP(f_k+h)$ and $\lP(f_k+h)$, respectively.
By Proposition~\ref{prop-Pf+h-Pf-Hf}, we have that $f_k$ and $f_k+h$ are resolutive and moreover $P(f_k+h)=Pf_k$.
Using the resolutivity of $f_k+h$, we have 
\[
\uP(f+h)-2^{-k}\le\uP(f_k+h)=\lP(f_k+h)\le\lP(f+h)+2^{-k}\quad\text{in }\Om,
\]
from which it follows that
\[
0\le\uP(f+h)-\lP(f+h)\le 2^{1-k}\quad\text{in }\Om.
\]
Letting $k\to\infty$ shows that $f+h$ is resolutive. In the same way, $f$ is resolutive.
Next, we have from \eqref{eq-fk-f} that
\[
P(f+h)-2^{-k}\le P(f_k+h)=Pf_k\le Pf+2^{-k} \quad\text{in }\Om,
\]
from which we get
\[
P(f+h)-P f\le 2^{1-k}.
\]
Similarly,
\[
Pf-P(f+h)\le 2^{1-k}.
\]
Letting $k\to\infty$ shows that $P(f+h)=Pf$.
\end{proof}

If $u$ is a bounded $\A$-harmonic function in $\Om$
such that 
\[
f(x)=\lim_{\Om\ni y\rightarrow x}u(y) \quad\text{for all }x\in\bdy\Om,
\]
then 
$u \in \U_f \cap \LL_f$.
Thus, 
\[
u\le \lP f \le \uP f \le u,
\]
and so $f$ is resolutive and $u=Pf$, see \cite[p.\ 169]{HKM}.
Using Theorem~\ref{thm-Pf+h-Pf}, we 
can now generalize this fact and deduce
the following uniqueness result.

\begin{cor}		\label{cor-unique}
Let $f\in C(\bdy\Om)$. Assume that $u$ is bounded and $\A$-harmonic in $\Om$ 
and that there is a set $E\subset\bdy\Om$ with $\Cpw(E)=0$ such that
\[
\lim_{\Om\ni y\to x}u(y)=f(x) \quad\text{for all }x\in\bdy\Om\setm E.
\]
Then $u=Pf$ in $\Om$.
\end{cor}

\begin{proof}
Add a sufficiently large constant to both $f$ and $u$, and then
rescale the new values of $f$ and $u$ so that $0\le f\le 1$ and $0\le u\le1$. 
Since $u$ is bounded and $\A$-harmonic in $\Om$, we have that $u\in\LL_{f+\chi_E}$ and $u\in\U_{f-\chi_E}$.
Thus, by Theorem~\ref{thm-Pf+h-Pf}, we get that
\[
u\le\lP(f+\chi_E)=Pf=\uP(f-\chi_E)\le u\quad\text{in }\Om. \qedhere
\]
\end{proof}

\begin{remark}
The word \emph{bounded} is essential for the above uniqueness result to hold. Otherwise it fails. For instance, the Poison kernel 
\[
\frac{1-|z|^2}{|1-z|^2}											
\]
with a pole at $1$ is a harmonic function in the unit disc
$B(0,1)\subset\C =\R^2$,  which is zero on $\bdy B(0,1)\setm\{1\}$.
\end{remark}

The Kellogg property
\cite[Theorem~9.11]{HKM}
together with Corollary~\ref{cor-unique}, yields the following uniqueness result.

\begin{thm}\label{thm-unique}
Let $f\in C(\bdy\Om)$. Then there exists a unique bounded
$\A$-harmonic function $u$ in $\Om$ such that
\begin{equation} \label{eq-unique}
\lim_{\Om\ni y\to x}u(y)=f(x) \quad\text{for q.e.\ }x\in\bdy\Om,
\end{equation}
moreover $u=Pf$.
\end{thm}

\begin{proof}
By the Kellogg property \cite[Theorem~9.11]{HKM} and Theorem~\ref{thm-Pf+h-Pf}, 
we have that $u=Pf$ satisfies \eqref{eq-unique}. 
On the other hand, if $u$ satisfies \eqref{eq-unique}, then Corollary~\ref{cor-unique} 
shows that $u=Pf$.
\end{proof}

\section{Quasicontinuous functions}
One of the useful properties of the Sobolev space $\Hp(\Om,w)$ is that every function 
in $\Hp(\Om,w)$ has a $(p,w)$-quasicontinuous representative which is unique 
upto sets of $(p,w)$-capacity zero, see 
\cite[Theorem~4.4]{HKM}.

\begin{defi} \label{def-qc}
A function \(v:\Om\to[-\infty,\infty]\) is \((p,w)\)-quasicontinuous in $\Om$ 
if for every $\eps>0$ there is an open set $G$ such that \(C_{p,w}(G)<\eps\) 
and the restriction of $v$ to $\Om\setm G$ is finite  valued and continuous.
\end{defi}

It follows from the outer regularity \eqref{eq-Cpw-outer} of
  $\Cpw$ that if $v$ is quasicontinuous and $\vb=v$ q.e.\ then $\vb$
  is also quasicontinuous.

Refining the techniques in Section~\ref{sect-Perron-res},
  we can obtain the following result.

\begin{thm} \label{thm-Sob-Pf=Hf}
Let $f:\R^n\to [-\infty,\infty]$ be a $(p,w)$-quasicontinuous function in $\R^n$ 
such that $f\in\Hp(\Om,w)$. 
Let $h:\bdy\Om\to [-\infty,\infty]$ be such that $h=0$ q.e.\ on $\bdy\Om$.
Then $f+h$ and $f$ are resolutive and $P(f+h)=Pf=Hf$.
\end{thm}

In $f+h$ we can interpret $\pm\infty\mp\infty$ arbitrarily
  in $[-\infty,\infty]$.
Before the proof of Theorem~\ref{thm-Sob-Pf=Hf}, we give the following
two lemmas which may be of independent interest.

\begin{lem}  \label{lem-Hf-f-qc}
Let $f$ be as in Theorem~\ref{thm-Sob-Pf=Hf}. Then its $\A$-harmonic extension $Hf$, 
extended by $f$ outside $\Om$, is $(p,w)$-quasicontinuous in $\R^n$.
\end{lem}

\begin{proof}
Define $v:=Hf-f$ and extend it by zero outside $\Om$.
Then $v\in\Hp_0(\Om,w)$.
By \cite[Theorem~4.5]{HKM}, there is a $(p,w)$-quasicontinuous function $\vb$ 
in $\R^n$ such that $\vb=v$ a.e.\ in $\Om$ and $\vb=0$ q.e.\ in the complement of $\Om$.
Recall that $Hf$ is a continuous function in $\Om$ and $f$ is assumed 
to be
$(p,w)$-quasicontinuous in $\Om$. 
This clearly means that $v$ is also $(p,w)$-quasicontinuous in $\Om$.
It then follows from \cite[Theorem~4.12]{HKM} that $v=\vb$ q.e.\ in $\Om$.
We know that $v=0$ outside the set $\Om$.
Thus, we can conclude that $\vb=v$ q.e.\ in $\R^n$.
Finally, by \eqref{eq-Cpw-outer},
since $\vb$ is $(p,w)$-quasicontinuous in $\R^n$, so is $v$ 
 and hence also $f+v$, which concludes the proof. 
\end{proof}

\begin{lem}		\label{lem-Hfj-Hf}
Let \(\{f_j\}_{j=1}^\infty\) be a decreasing sequence of functions in $\Hp(\Om, w)$ such that $f_j\to f$ in $\Hp(\Om,w)$.
Then the sequence $Hf_j$ decreases to $Hf$ in $\Om$. 
\end{lem}
\begin{proof}
	
By the comparison principle \cite[Lemma~3.18]{HKM}, we have for all $j=1,2,\dots$\,, 
\[
u_j:=Hf_j\ge Hf_{j+1}\ge\cdots\ge Hf\quad\text{in } \Om.
\]
Thus $u(x)=\lim_{j\to\infty}u_j(x)$ exists for all $x\in\Om$ and $u(x)\ge Hf(x)$.
Note that $Hf$ is continuous in $\Om$ and so the sequence $u_j$ is locally bounded from below in $\Om$.
By \cite[Theorem~3.77]{HKM}, $u$ is a supersolution in $\Om$.
Similarly, \cite[Theorem~3.75]{HKM} applied to $-u_j$ shows that $u$ is a subsolution.
Hence $u$ is a solution of \eqref{Div-A} in $\Om$,
see \cite[bottom p.~58]{HKM}.

To conclude the proof, we need to show that $u-f\in\Hp_0(\Om,w)$.
We know that $u_j-f_j\to u-f$ pointwise a.e.\ and $u_j-f_j\in\Hp_0(\Om,w)$. 
Because of \cite[Lemma~1.32]{HKM}, it is sufficient to show that $u_j-f_j$ is a bounded sequence in $\Hp(\Om,w)$.

Using the Poincar\'e inequality~\cite[(1.5)]{HKM} we have
\begin{align*}
\|u_j-f_j\|_{\Hp(\Om,w)}&\le C_{\Om}\biggl(\int_\Om|\grad u_j-\grad f_j|^p w\,dx\biggr)^{1/p}\\
&\le C_{\Om} \biggl(\int_\Om|\grad u_j|^p w\,dx\biggr)^{1/p}
+C_{\Om} \biggl(\int_\Om|\grad f_j|^p w\,dx\biggr)^{1/p},
\end{align*}
where $C_\Om$ is a constant which depends on $\Om$.
Since $u_j$ is a solution and 
$\A$ satisfies the ellipticity conditions~\eqref{ellip-conds},
testing~\eqref{Int-A} with $\phi=u_j-f_j$ yields
\[
\biggl(\int_\Om|\grad u_j|^p w\,dx\biggr)^{1/p}
\le C \biggl(\int_\Om|\grad f_j|^p w\,dx\biggr)^{1/p},
\]
where $C$ is a constant depending on
the structure constants
$\al$ and $\be$ in \eqref{ellip-conds}.
Therefore,
\[
\|u_j-f_j\|_{\Hp(\Om,w)}\le C'\biggl(\int_\Om|\grad f_j|^p w\,dx\biggr)^{1/p}
\le C'\|f_j\|_{\Hp(\Om,w)}\le M<\infty,
\]
since the sequence $\{f_j\}_{j=1}^\infty$ is bounded in $\Hp(\Om,w)$. 
This shows that $u_j-f_j$ is bounded in $\Hp(\Om,w)$.
 Consequently, by \cite[Lemma~1.32]{HKM}, 
$u-f\in\Hp_0(\Om,w)$ and $u=Hf$ by uniqueness, cf.\ \cite[Theorem~3.17]{HKM}.
\end{proof}

We now prove Theorem~\ref{thm-Sob-Pf=Hf} and refer the reader to closely look 
at the proof of Proposition~\ref{prop-Pf+h-Pf-Hf} to fill in details where needed.

\begin{proof}[Proof of Theorem~\ref{thm-Sob-Pf=Hf}]
First assume that $f\ge0$ and so $Hf\ge0$.
Define $u:=Hf$ extended by $f$ outside $\Om$.
By Lemma~\ref{lem-Hf-f-qc}, $u$ is $(p,w)$-quasicontinuous  in $\R^n$. 
Let $\{U_k\}_{k=1}^\infty$ be a decreasing sequence of bounded open sets in $\R^n$ 
such that $\Cpw(U_k)<2^{-kp}$,
$h=0$ outside $U_k$ and $u$ restricted to $\R^n\setm U_k$ 
is continuous.
Let $u_j$ be the lsc-regularized solution of the
obstacle problem with the obstacle 
and boundary data $f_j=u+\psi_j$, where $\psi_j$ are as in Lemma~\ref{lem-small-Cp}, 
$j=1,2,\cdots$\,.
As in the proof of Proposition~\ref{prop-Pf+h-Pf-Hf} we get
\begin{equation} \label{eq-uj-m*}
u_j\geq m\quad\text{everywhere in }U_{j+m}\cap\Om.
\end{equation}

Let $\eps>0$ and $x\in\bdy\Om$ be arbitrary. 
If $x\in\bdy\Om\setm U_{j+m}$, then by quasicontinuity, $u$ restricted to $\R^n\setm U_{j+m}$ is continuous at $x$. 
Thus, there is a neighbourhood $V_x$ of $x$ such that 
\[
Hf(y)=u(y)\ge u(x)-\eps=f(x)-\eps=(f+h)(x)-
\eps\quad \text{for all }y\in (V_x\cap\Om)\setm U_{j+m}.
\]
Since $\psi_j\ge0$, we get $f_j(y)\ge u(y)=Hf(y)$ and so,
\begin{equation}  \label{eq-uj-f+h}
  u_j(y)\ge f_j(y)\ge (f+h)(x)-\eps         \quad\text{for a.e.\ }y\in (V_x\cap\Om)\setm U_{j+m}.
\end{equation}
If $x\in U_{j+m}$, let $V_x=\emptyset$. Then by \eqref{eq-uj-m*} and \eqref{eq-uj-f+h}, we get for all $x\in\bdy\Om$,
\[
u_j(y)\ge\min\{(f+h)(x)-\eps,m\}\quad\text{for a.e. }y\in(V_x\cup U_{j+m})\cap\Om
\]
Since $u_j$ is lsc-regularized, we have  
\[ 
u_j(y)\ge\min\{(f+h)(x)-\eps,m\}\quad\text{for all }y\in
(V_x\cup U_{j+m})\cap\Om,
\] 
and consequently, \eqref{eq-min-f+h-m} follows.
Letting $\eps\to 0$ and $m\to\infty$, we conclude that $u_j\in\U_{f+h}(\Om)$.
Continuing as in Proposition~\ref{prop-Pf+h-Pf-Hf}, we can conclude that 
\begin{equation}	\label{Ineq-Hf-Pf+h-4}
  \uP (f+h)  \le Hf\quad\text{in }\Om
\end{equation}
holds for all quasicontinuous 
$f:\R^n\to [-\infty,\infty]$ in $\Hp(\Om,w)$
  that are nonnegative (or merely bounded form below).

Now if $f\in\Hp(\Om,w)$ is arbitrary, then by \eqref{Ineq-Hf-Pf+h-4}
together with Lemma~\ref{lem-Hfj-Hf}
we have that
\[
\uP(f+h) \le\lim_{k\to-\infty}\uP(\max\{f,k\}+h)
          \le\lim_{k\to-\infty}H\max\{f,k\}=Hf
\quad\text{q.e.\ in }\Om.
\]
Thus, \eqref{Ineq-Hf-Pf+h-4} holds for any $f\in\Hp(\Om,w)$ and applying it 
to $-f$ and $-h$
together with the inequality
$\lP(f+h)\le\uP(f+h)$,
concludes the proof.
\end{proof}

Unlike for continuous boundary data in Theorem~\ref{thm-Pf+h-Pf}, for quasicontinuous boundary data it is
in general impossible to have
$\lim_{\Om\ni y\to x}Pf(y)=f(x)$ for q.e.\ $x\in\bdy\Om$,
see Example~\ref{ex-qcont} below.
However, we get the following uniqueness result
as a consequence of 
Theorem~\ref{thm-Sob-Pf=Hf}.

\begin{cor}  \label{cor-unique-qcont}
Let $f:\R^n\to [-\infty,\infty]$ be a $(p,w)$-quasicontinuous function in $\R^n$ 
such that $f\in\Hp(\Om,w)$.
Assume that $u$ is a bounded $\A$-harmonic function in $\Om$ and that there is a set 
$E\subset\bdy\Om$ with $\Cpw(E)=0$ such that
\[
\lim_{\Om\ni y\to x}u(y)=f(x) \quad\text{for all }x\in\bdy\Om\setm E.
\]
Then $u=Pf$.
\end{cor}

\begin{proof}
Since
$u$ is a bounded $\A$-harmonic function in $\Om$, we have that
$u\in\LL_{f+\infty\chi_E}$ and $u\in\U_{f-\infty\chi_E}$.
Thus by Theorem~\ref{thm-Sob-Pf=Hf}, we get that
\[
u\le\lP(f+\infty\chi_E)=Pf=\uP(f-\infty\chi_E)
\le u \quad\text{in } \Om. \qedhere
\]
\end{proof}

The following example shows that
in many situations there is 
a bounded quasicontinuous function $f \in \Hp(\R^n,w)$ 
such that no function $u$ satisfies
\[
\lim_{\Om\ni y\to x}u(y)=f(x)
  \quad \text{for q.e.\ } x \in \bdy \Om.
\]
In particular it is impossible for the Perron solution $Pf$
to attain these quasicontinuous boundary data q.e.

\begin{example} \label{ex-qcont}
Assume that $\bdy \Om$ contains a dense countable sequence $\{x_j\}_{j=1}^\infty$
of points with $\Cpw(\{x_j\})=0$, $j=1,2,\ldots$\,.
As $\Om$ is bounded it follows from
\cite[Corollary~2.39 and Lemma~2.46]{HKM} that $\Cpw(\bdy \Om)>0$.
Using  \eqref{eq-Cpw-outer}, we can then find $r_j>0$ so small that
$\Cpw(B(x_j,r_j))<3^{-j} \Cpw(\bdy \Om)$, $j=1,2,\ldots$\,.

By \cite[Corollary~2.39]{HKM}, each $x_j$ has zero variational
$(p,w)$-capacity, and hence, by the definition  of
the variational capacity \cite[p.~27]{HKM} there
is $f_j \in C_0^\infty(B(x_j,r_j))$ such that $f_j(x_j)=1$
and $\|f_j\|_{\Hp(\R^n,w)} < 2^{-j}$.
Then
\[
f:=\sum_{j=1}^\infty \max\{f_j,0\} \in \Hp(\R^n,w).
\]
Since the partial sums of $f$ are continuous and coincide
  with $f$ outside the open sets $\bigcup_{j\ge k} B(x_j,r_j)$,
  $k=1,2,\ldots$\,, with arbitrarily small $(p,w)$-capacity,
we see that $f$ is quasicontinuous.
For each $j$ there is $r_j'<r_j$ such that $f_j \ge \tfrac12$ in $B(x_j,r_j')$.
Thus
\[
f\ge\tfrac12
\text{ in } G'=\bigcup_{j=1}^\infty B(x_j,r_j')
\quad \text{and} \quad 
f = 0 \text{ outside  } G=\bigcup_{j=1}^\infty B(x_j,r_j).
\]
Note that $\Cpw(G) < \sum_{j=1}^\infty 3^{-j} \Cpw(\bdy \Om) < \Cpw(\bdy \Om)$.
Also let
\[
     S=\{x \in \bdy \Om: \text{there is $r>0$ such that 
          $\Cpw(B(x,r) \cap \bdy \Om)=0$}\},
\]
which is the largest relatively open subset of $\bdy \Om$ with
$\Cpw(S)=0$.

Finally, assume that $u : \Om \to \R$ is such that
\begin{equation} \label{eq-ut}
\ut(x):=\lim_{\Om\ni y\to x}u(y)=f(x)
  \quad \text{for q.e.\ } x \in \bdy \Om.
\end{equation}
In particular $\ut\ge\tfrac12$
q.e.\ in $G' \cap \bdy \Om$, and thus in a dense
subset of $\bdy \Om \setm S$.
It follows that
\[
\limsup_{\Om\ni y\to x}u(y) \ge \tfrac12
  \quad \text{for all } x \in \bdy \Om \setm S.
\]
But this violates the assumption that $\ut(x)=f(x)=0$
q.e.\ in $\bdy \Om \setm G$, since 
$\Cpw(\bdy \Om \setm G)>0$.
Hence there is no function $u$ satisfying \eqref{eq-ut}.

Replacing $f$ by $\min\{f,1\}$ yields a similar bounded counterexample.
\end{example}

\end{document}